\documentclass[reqno]{amsart}
\usepackage{amssymb,amscd,verbatim, amsthm,graphics, color,latexsym,amsmath,multicol}
\usepackage{fancybox}
\usepackage{longtable}

\usepackage[all]{xy}

\newcommand \C[1]{{\mathcal #1}}

\newcommand \wti[1]{{\widetilde {#1}}}

\newcommand\fg{\mathfrak g}

\newcommand \bC{{\mathbb C}}

\newcommand \bH{{\mathbb H}}
\newcommand \bR{{\mathbb R}}
\newcommand \bZ{{\mathbb Z}}

\newcommand \bQ{{\mathbb Q}}

\newcommand\ep{{\epsilon}}

\newcommand\om{{\omega}}
\newcommand\al{{\alpha}}

\newcommand\fh{{\mathfrak h}}

\newtheorem{theorem}{Theorem}[subsection]

\newtheorem{corollary}{Corollary}[subsection]
\newtheorem{lemma}{Lemma}[subsection]
\newtheorem{proposition}{Proposition}[subsection]
\newtheorem{definition}{Definition}[subsection]
\newtheorem{remark}{Remark}[subsection]

\newcommand\Hom{\operatorname{Hom}}

\newcommand\Ind{\operatorname{Ind}}

\newcommand\im{\operatorname{im}}

\newcommand\triv{\mathsf{triv}}
\newcommand\sgn{\mathsf{sgn}}
\newcommand\refl{\mathsf{refl}}

\newcommand\Pin{\mathsf{Pin}}

\numberwithin{equation}{subsection}

\begin{document}

\title{Dirac cohomology of one $K$-type representations}

\author{Dan  Ciubotaru}
        \address[D. Ciubotaru]{Dept. of Mathematics\\ University of
          Utah\\ Salt Lake City, UT 84112}
        \email{ciubo@math.utah.edu}

\author{Allen Moy}
\address[A. Moy]{Dept. of Mathematics\\Hong Kong University of Science
  and Technology\\Hong Kong}
\email{amoy@ust.hk}

\thanks{This paper was partly written while the first author visited Hong
  Kong University of Science and Technology. He thanks
  Xuhua He and the Department of Mathematics for  their invitation and
  hospitality. The authors were supported in part by
NSF-DMS 0968065 and Hong Kong Research Grants Council grant CERG {\#}602408.}

\begin{abstract}The smooth hermitian representations of a split reductive $p$-adic
  group whose restriction to a maximal hyperspecial compact subgroup
  contain a single $K$-type with Iwahori fixed vectors have been
  studied in \cite{BM2} in the more general setting of modules for graded
  affine Hecke algebras with parameters. We show that every such one
  $K$-type module has nonzero Dirac cohomology (in the sense
  of \cite{BCT}), and use Dirac operator techniques to determine the
  semisimple part of the Langlands parameter for these modules, thus
  completing their classification.
\end{abstract}

\maketitle

\section{Introduction}\label{sec:1}
The category of smooth representations of a reductive $p$-adic group
generated by their vectors fixed under an Iwahori subgroup
is equivalent to the category of modules over the Iwahori-Hecke
algebra \cite{Bo}. Furthermore, the category of Iwahori-Hecke algebra
modules is equivalent to a product of categories of certain graded
affine Hecke algebra modules \cite{L}. It is known that these
equivalences induce bijections between the unitary
representations in the corresponding categories \cite{BM,BM1}. 

An interesting class of unitary representations are those which have a single
$K$-type with Iwahori fixed vectors, or in terms of Borel's
equivalence of categories, the Iwahori-Hecke algebra modules whose
restrictions to the finite Hecke algebra are irreducible. These
representations are expected to be automorphic, for example, the Speh
representations for $GL(n,\bQ_p)$ (\cite{Ta}) are such one $K$-type
representation.

 In terms of the corresponding graded affine Hecke algebra $\bH$
 (Definition \ref{d:graded}), we are interested in the one $W$-type
 modules, i.e., unitary
modules of $\bH$ whose restrictions to the group algebra $\bC[W]$ of
the Weyl group are irreducible. In \cite{BM2}, it was determined which
irreducible Weyl group representations support an action of hermitian
(hence necessarily unitary)
$\bH$-module with respect to the natural $*$-operation
(\ref{e:star}). In particular, it is shown that whenever this is
possible, there is a unique way to define the action of $\bH$. In this
paper, we first provide a simpler, uniform proof for this fact, Proposition \ref{p:2.1}; the
idea is similar in spirit with the classical argument from real groups used
to show that there are no nontrivial unitary representations of a
simple real group, e.g. \cite[Corollary 2.3]{Kn} for $SL(2,\bR)$.

In \cite{BCT}, a Dirac operator for graded Hecke algebras was
introduced and a theory of Dirac cohomology for $\bH$-modules was
proposed. We show here that every one $W$-type $\bH$-module has
nonzero Dirac cohomology, in fact the Dirac operator is identically
zero on such modules, Proposition \ref{p:3.2}. This fact, combined
with certain calculations of tensor products of representations of the
pin double cover of the Weyl group $\wti W$ (defined in section
\ref{sec:2.2}), allows us to compute explicitly the Dirac cohomology
of these modules in all cases. Vogan's conjecture, in the setting of
Hecke algebras proved in \cite{BCT}, essentially says that the Dirac
cohomology determines the central character of the $\bH$-module. Thus,
as a corollary of our calculations, we find the central character of
the one $W$-type $\bH$-modules, which completes their identification
in the Langlands classification. The explicit results are listed in
sections \ref{sec:3.3}-\ref{sec:3.5}.

\section{Dirac operator for graded Hecke algebras}\label{sec:2}

In this section, we review the construction and properties of the Dirac
operator from \cite{BCT} and the relation between spin representations
of the Weyl group and nilpotent orbits from \cite{C}.

\subsection{The root system}\label{sec:2.1} Let $V$ be a finite
dimensional Euclidean vector space, with inner product $(~,~): V\times
V\to\bR$. A (reduced) root system $\Phi$ is a finite set $\Phi\subset
V\setminus\{0\}$ such that:
\begin{enumerate}
\item $\Phi$ spans $V$;
\item for every $\alpha\in\Phi$, the reflection through the hyperplane
  perpendicular on $\al$, $s_\al: V\to V$,
  $s_\al(v)=v-2\frac{(\al,v)}{(\al,\al)}\al$, preserves $\Phi$, i.e.,
  $s_\al(\Phi)\subset\Phi$;
\item if $\al\in\Phi,$ then $2\al\notin\Phi$.
\end{enumerate}
Let $W$ be the subgroup of $GL(V)$ generated by
$\{s_\al:\al\in\Phi\}.$ In fact, since $(~,~)$ is $W$-invariant, $W$
is a (finite) subgroup of $O(V).$ 

Choose a system $\Phi^+\subset \Phi$ of positive roots, and let $\Pi$ be
a basis of $\Phi^+,$ the set of simple roots. 
The group $W$ admits a Coxeter presentation
\begin{equation}\label{e:CoxW}
W=\langle s_\al,\al\in\Pi: (s_\al s_\beta)^{m(\al,\beta)}=1\rangle,
\end{equation}
where $m(\al,\al)=1$ and $m(\al,\beta)\in\bZ$ is such that the angle
between $\al$ and $\beta$ equals $\pi/m(\al,\beta)$ when $\al\neq\beta.$

\subsection{The pin cover of $W$}\label{sec:2.2} Let $C(V)$ be the
Clifford algebra of $V,(~,~)$. This is the real associative algebra
generated by $\{v\in V\}$ subject to the relations
\begin{equation}
v_1 v_2+v_2 v_1=-2(v_1,v_2),\quad v_1,v_2\in V.
\end{equation}
If one assigns degree one to the elements $v\in V$, then $C(V)$ is
naturally a filtered algebra whose associated graded algebra is
$\wedge V.$ In particular, $C(V)$ has a $\bZ/2\bZ$-grading
$C(V)=C(V)_{\text{even}}+C(V)_{\text{odd}}$ given by the degree mod
$2$. Let $\ep$ be the automorphism of $C(V)$ which is $1$ on
$C(V)_{\text{even}}$ and $-1$ on $C(V)_{\text{odd}}.$ Let $\ ^t:
C(V)\to C(V)$ be the anti-involution given by $v^t=-v$ for $v\in V.$
Define the pin group:
\begin{equation}
\Pin(V)=\{a\in C(V): \ep(a) V a^{-1}\subset V, a^t=a^{-1}\},
\end{equation} 
a central $\bZ/2\bZ$-extension of $O(V)$ with the projection map $p:
\Pin(V)\to O(V)$, $p(a)(v)=\ep(a) v a^{-1}$, $a\in C(V),$ $v\in V.$
Since $W\subset O(V),$ define
\begin{equation}
\wti W=p^{-1}(W)\subset \Pin(V).
\end{equation}
Analogous to (\ref{e:CoxW}), $\wti W$ also admits a Coxeter-like
presentation:
\begin{equation}\label{e:CoxWtilde}
\wti W=\langle z, \wti s_\al, \al\in\Pi: z^2=1, \wti s_\al^2=z, (\wti s_\al\wti s_\beta)^{m(\al,\beta)}=z\rangle.
\end{equation}
The embedding $\wti W\subset \Pin(V)$ is given by
\begin{equation}
z\mapsto -1,\quad \wti s_\al\mapsto \frac 1{|\al|}\al,\ \al\in\Phi^+,
\end{equation}
where $|\al|=\sqrt{(\al,\al)}.$

\smallskip

When $\dim V$ is even, the Clifford algebra $C(V)$ is a central simple
algebra, and therefore has a unique simple complex module $S$ of
dimension $2^{\dim V/2}.$
When $\dim V$ is odd, the center of $C(V)$ is two dimensional. The
subalgebra $C(V)_{\text{even}}$ is central simple and has a unique
simple complex module which can be extended in two nonisomorphic ways
to $C(V).$ Thus $C(V)$ has two simple modules $S^+$ and $S^-$ of
dimension $2^{(\dim V-1)/2}$.

In the sequel, we refer to any one of $S, S^+, S^-$ as a spin
module. Notice that since $\wti W$ generates $C(V),$ the restriction
of a spin module $S$ to $\wti W$ remains irreducible.

\subsection{Adjoint nilpotent orbits}\label{s:2.3} We assume now that
the root system $\Phi$ is crystallographic, i.e., it satisfies, in
addition, the axiom
\begin{enumerate}
\item[(3)] for every $\al,\beta\in \Phi$, $2\frac{(\al,\beta)}{(\al,\al)}\in\bZ.$
\end{enumerate}
Let $\Phi^vee\subset V^*$ be the dual root system to $\Phi\subset V.$
More precisely, $\Phi^\vee$ is the set of coroots $\al^\vee,$ where
$\al^\vee(v)=\frac 2{(\al,\al)} (\al,v)$, $v\in V.$

Let $\fg$ be the complex semisimple Lie algebra with root system
$\Phi^\vee.$ We identify a Cartan subalgebra $\fh$ of $\fg$ with
$V$. Let $\C N$ denote the nilpotent cone in $\fg,$ i.e., the set of
ad-nilpotent elements of $\fg.$ The adjoint group $G$ whose Lie
algebra is $\fg$ acts via the adjoint action on $\C N$ with finitely
many orbits. A classical result of Kostant is that there is a
bijection between the $G$-orbits in $\C N$ and the $G$-conjugacy
classes of Lie triples
\begin{equation}\label{e:Lietriple}
\{e,h,f\}:\ [h,e]=2e,\ [h,f]=-2f,\ [e,f]=h.
\end{equation}
Let $\C B$ be the flag variety, the variety of all Borel subalgebras
of $\fg.$ For $e\in \C N$, let $\C B_e$ be the subvariety of Borel
subalgebras which contain $e$. Let $A(e)$ be the component group of
the centralizer of $e$ in $G.$ 

Springer theory defines an action of $W\times A(e)$ on the cohomology
groups $H^\bullet (\C B_e,\bC)$ such that:
\begin{enumerate} 
\item for every $e\in \C N$ and $\phi$ an irreducible ${A(e)}$-representation, the $\phi$-isotypic
  component of the top cohomology group $H^{2 d_e}(\C B_e,\bC)^\phi$
  is an irreducible $W$-representation or zero; denote $\widehat
  {A(e)}_0$, the set of irreducible $A(e)$-representation for which
  $\sigma(e,\phi):=H^{2 d_e}(\C B_e,\bC)^\phi\neq 0$;
\item the set $\{\sigma(e,\phi): e\in G\backslash \C N,\phi\in\widehat
  {A(e)}_0\}$ equals the set of irreducible representations of $W$.
\end{enumerate}

\subsection{Genuine $\wti W$-representations and nilpotent orbits}\label{sec:2.4} Suppose
$k:\Phi^+\to\bR,$ $k(\al)=k_\al$ is a $W$-invariant function. Define (\cite{BCT})
\begin{equation}\label{e:casimirW}
\Omega_{\wti W,k}=z (\sum_{\al\in\Phi^+} k_\al |\al| \wti s_\al)^2.
\end{equation}
This is an element in the center of $\bC[\wti W].$ Set
\begin{equation}
\C N_{\text{sol}}=\{e\in\C N: Z_G(e)^0\text{ is solvable}\},
\end{equation}
a $G$-invariant subset of $\C N.$

Call an irreducible $\wti W$-representation $\wti\sigma$ genuine if
$\wti\sigma(z)=-1$, i.e., if $\wti\sigma$ does not descend to a
representation of $W.$ Denote by $\text{Irr}_{\text{gen}}\wti W$ the
set of irreducible genuine $\wti W$-representations.

\begin{theorem}[\cite{C}] Suppose the parameter function is $k=1.$
  There exists a surjective map
\begin{equation}
\Psi: \text{Irr}_{\text{gen}} \wti W\to G\backslash \C N_{\text{sol}},
\end{equation}
with the following properties:
\begin{enumerate}
\item If for $\wti \sigma\in \text{Irr}_{\text{gen}} \wti W$,
  $\Psi(\wti\sigma)=G\cdot e, $ where $e\in \C N_{\text{sol}},$ then
  $\wti\sigma(\Omega_{\wti W,1})=(h,h),$ where $\{e,h,f\}$ is a Lie
  triple for $e$ with $h\in V.$
\item Given $e\in \C N_{\text{sol}},$ if $\wti\sigma\in
  \Psi^{-1}(G\cdot e)$, then there exists $\phi\in\widehat{A(e)}_0$
  and a spin module $S$ such that $\wti\sigma$ occurs in
  $\sigma(e,\phi)\otimes S.$
\item Given $e\in \C N_{\text{sol}}$, $\phi\in\widehat {A(e)}_0$, and
  $S$ a spin module, there exists $\wti\sigma\in \Psi^{-1}(G\cdot e)$
  such that $\wti\sigma$ occurs in $\sigma(e,\phi)\otimes S.$
\end{enumerate}
\end{theorem}

There is an explicit description of the map $\Psi$ for every simple root
system \cite{C}.

\subsection{Graded affine Hecke algebra}\label{sec:2.5}
Let $V_\bC$ denote the complexification of $V$.
\begin{definition}[\cite{L}]\label{d:graded} The graded affine Hecke
  algebra associated to the root system $\Phi\subset V$ and parameter
  function $k:\Phi^+\to\bR$ is the unique associative complex algebra
  with identity generated by $\{t_w: w\in W\}$ and $\{\omega:\omega\in
  V_\bC\}$ with relations:
\begin{enumerate}
\item $t_w t_{w'}=t_{ww'},$ $w,w'\in W$;
\item $\omega\omega'=\omega'\omega,$ $\omega,\omega'\in V_\bC$;
\item $\omega t_{s_\al}-t_{s_\al} s_\al(\omega)=k_\al
  \frac{\omega-s_\al(\omega)}\al,$ $\omega\in V_\bC,$ $\al\in \Pi.$
\end{enumerate}
\end{definition}
The center of $\bH$ is $Z(\bH)=S(V_\bC)^W,$ where $S(V_\bC)$ denotes the
symmetric algebra of $V_\bC.$ Consequently, $\bH$ is finite over its
center, hence every irreducible $\bH$-module is finite dimensional,
and the center of $\bH$ acts by characters in the irreducible
modules. The central characters are parameterized by $W\backslash
V_\bC^*.$ If $\nu\in V_\bC^*,$ write $\chi_\nu$ for the central
character parameterized by $W\nu.$

In \cite{BCT}, the Casimir element $\Omega$ of $\bH$ was
introduced: if $\{\om_i\}$  is an orthonormal basis of $V$, set 
\begin{equation}
\Omega=\sum_i \om_i^2\in S(V_\bC)^W.
\end{equation}
If $(\pi,X)$ is an irreducible $\bH$-module with central character
$\chi_\nu$, then $\pi(\Omega)=(\nu,\nu).$

Let $*$ be the conjugate linear anti-involution of $\bH$ defined on
generators by \cite{BM1}:
\begin{equation}\label{e:star}
t_w^*=t_{w^{-1}},\quad \omega^*=-t_{w_0} w_0(\om) t_{w_0},\quad w\in
W, \ \om\in V,
\end{equation}
where $w_0$ is the longest Weyl group element. We say that an
$\bH$-module $(\pi,X)$ is $*$-hermitian if $X$ has a hermitian form
$\langle~,~\rangle$ which is $*$-invariant, i.e.,
\begin{equation}
\langle\pi(h)x,y\rangle=\langle x,\pi(h^*)y\rangle.
\end{equation}
If the form is positive definite, we say that $(\pi,X)$ is $*$-unitary.

\subsection{The Dirac operator}\label{sec:2.6} For every $\om\in V$,
define
\begin{equation}\label{e:omtilde}
\wti\om=\frac 12 (\om-\om^*)=\om-p_\om,\text{ where }p_\om=\frac 12\sum_{\al\in\Phi^+} k_\al
\frac{\om-s_\al(\om)}\al t_{s_\al}.
\end{equation}
By definition, $\wti\om^*=-\wti\om.$ In addition, these elements have
the known properties (e.g., \cite{BCT}):
\begin{equation}\label{e:propomtilde}
t_w\wti\om t_{w^{-1}}=\wti{w(\om)},\quad [\wti\om_1,\wti\om_2]=-[p_{\om_1},p_{\om_2}].
\end{equation}
Let $\{\om_i\}$ be an orthonormal basis of $V.$ The Dirac element
defined in \cite{BCT} is
\begin{equation}\label{e:dirac}
\C D=\sum_{i} \wti\om_i\otimes\om_i\in \bH\otimes C(V).
\end{equation}
It is independent of the choice of basis and its square is
\begin{equation}\label{e:squaredirac}
\C D^2=-\Omega\otimes 1+\frac 14\Delta_{\wti W}(\Omega_{\wti W,k}),
\end{equation}
where $\Delta_{\wti W}$ is the diagonal embedding of $\bC[\wti W]$
into $\bH\otimes C(V)$ extending $\Delta_{\wti W}(\wti w)=p(\wti
w)\otimes \wti w.$

If $(\pi,X)$ is an $\bH$-module, and $S$ is a spin $C(V)$-module, left
action by $\C D$ 
defines a Dirac operator
\begin{equation}\label{e:diracop}
D_X: X\otimes S\to X\otimes S.
\end{equation}
When $X$ is $*$-hermitian, $D_X$ is a self-adjoint operator with
respect to the tensor product form (the hermitian form on $S$ being
the natural one). Moreover, $D_X$ is $\sgn$ $\wti W$-invariant.

\begin{definition}\label{d:diracology}
The Dirac cohomology of $X$ (with respect to $S$) is $H^D(X)=\ker
D_X/\ker D_X\cap \im D_X,$ a $\wti W$-representation. 
\end{definition}

An analogue of the algebraic version of Vogan's conjecture from real
groups, proved in this setting in \cite{BCT}, says that there exists
an algebra homomorphism
\begin{equation}\label{e:zeta}
\zeta: Z(\bH)\to \bC[\wti W]^{\wti W},
\end{equation}
such that for every $z\in Z(\bH)$, there exists a unique $a\in
\bH\otimes C(V)$ with the property
\begin{equation}\label{e:algVogan}
z\otimes 1=\Delta_{\wti W}(\zeta(z))+\C D a+a\C D.
\end{equation}
For every $\wti\sigma\in \text{Irr}_{\text{gen}} \wti W,$ define a
central character
$\chi^{\wti\sigma}: Z(\bH)\to \bC$ by
$\chi^{\wti\sigma}(z)=\wti\sigma(\zeta(z)),$ $z\in Z(\bH).$ As
explained in \cite{BCT}, (\ref{e:algVogan}) implies the following
result.

\begin{theorem}[\cite{BCT}]\label{t:vogan} Let $X$ be an irreducible
  $\bH$-module , and
  suppose the irreducible $\wti W$-representation $\wti \sigma$ occurs
  in $H^D(X)$ (in particular $H^D(X)\neq 0$).  Then the central
  character of $X$ is $\chi^{\wti\sigma}.$
\end{theorem}

In light of this theorem, it is necessary to describe the central
characters $\chi^{\wti\sigma}$ more explicitly. When the parameter
function $k$ is identically $1$, it turns out that
\begin{equation}\label{e:identcc}
\chi^{\wti\sigma}=\frac 12 h, \text{ where $h$ is a middle element for
  $e$ when } \Psi(\wti\sigma)=G\cdot e.
\end{equation}
When $W$ is of type $B_n$, every irreducible $\wti W$-representation
is of the form $\wti\sigma=(\sigma\times 0)\otimes S,$ where $\sigma$
is a partition of $n$, $\sigma\times 0$ labels the irreducible
$W$-representation in the bipartition notation, and $S$ is a spin
$C(V)$-module. For such a $\wti\sigma$, the character
$\chi^{\wti\sigma}$ is described in \cite{C} by the following
combinatorial procedure. Suppose the parameter function $k$ takes the
value $k_s$ on the short simple root, and $k_\ell$ on the long simple
roots of type $B_n.$ In the left justified decreasing Young diagram of shape
$\sigma,$ label each box starting with $k_s$ in the upper left corner,
then increasing by $k_\ell$ to the right and decreasing by $k_\ell$
down. The entries of the resulting Young tableau, viewed in
$\bR^n\cong V$ form the central character $\chi^{\wti \sigma}.$ 

When $\bH$ is of type $G_2$ or $F_4$ with arbitrary parameters $k$,
the lists of central characters $\chi^{\wti\sigma}$ are given in \cite{COT}.

\section{One $W$-type modules}\label{sec:3}
Recall from the introduction that a simple $\bH$-module is called a
one $W$-type module if it is $*$-hermitian and its restriction to
$\bC[W]$ is irreducible. In \cite{BM2}, a classification of one $W$-type $\bH$-modules is obtained. We simplify
one of the arguments in \cite{BM2} and show that every one $W$-type
module has nonzero Dirac cohomology.

\subsection{Action of $\wti\om$} \label{sec:3.1}
The following proposition is implicitly proved in \cite{BM2}, under
an assumption about the restriction of the $W$-type to maximal
parabolic subgroups. We give here a different proof that works in
general and  can be regarded as an analogue of the well-known argument
which proves that the only irreducible finite dimensional unitary representation of
$SL(2,\bR)$ is trivial, e.g., \cite{Kn}.

\begin{proposition}\label{p:2.1}
Let $(\pi,X)$ be a one $W$-type
$\bH$-module. Then $\pi(\wti\om)=0$ for all $\om\in V.$ In particular,
an irreducible $W$-representation $\sigma$ can be extended to a one
$W$-type $\bH$ module if and only if
$[\sigma(p_{\om_1}),\sigma(p_{\om_2})]=0,$ for every $\om_1,\om_2\in V.$
\end{proposition}

\begin{proof} 
By definition, $(\pi,X)$ has a $*$-invariant hermitian form $\langle~,~\rangle_X^*$. Since $X$
is a one $W$-type module, this form can be normalized so that it is
positive definite. Since $\wti\om^*=-\wti\om$ for $\om\in V$, it follows that
$\pi(\wti\om)$ is a skew-symmetric operator on
$X,\langle~,~\rangle_X^*$. In particular, $\pi(\wti\om)$ is
diagonalizable and acts with
purely imaginary eigenvalues on $X.$

Consider now the conjugate linear anti-involution $\bullet$ of $\bH$
studied in \cite{BC}, defined on
generators by 
\begin{equation}
t_w^\bullet=t_{w^{-1}},\quad \om^\bullet=\om,\quad w\in W,\ \om\in V.
\end{equation}
From (\ref{e:omtilde}), it is immediate that $\wti\om^\bullet=\wti\om$
for every $\om\in V$. The two anti-involutions $*$ and $\bullet$ are
related via
\begin{equation}
h^\bullet=t_{w_0}\C A(h^*) t_{w_0},\quad h\in \bH,
\end{equation}
where $\C A$ is the automorphism of $\bH$ defined by $\C A(t_w)=t_{w_0
w w_0}$ and $\C A(\om)=-w_0(\om),$ $w\in W,$ $\om\in V.$ Let $(\pi^{\C
A},X)$ be the $\C A$-twist of the module $(\pi,X)$: $\pi^{\C A}(h)
x=\pi(\C A(h))x,$ $h\in\bH,$ $x\in X.$ 

An easy observation is that if $X$ is a one $W$-type with central
character $\chi_\nu$, then necessarily $\nu\in V$, i.e., $X$ has real
central character. (Because otherwise, $X$ would be unitarily induced
from a module with real central character of a parabolic subalgebra of
$\bH$ and then its restriction to $W$ would not be irreducible.)
Then by the classification of
$*$-hermitian irreducible modules \cite{BM1}, $(\pi,X)\cong (\pi^{\C
  A},X)$. Let $\kappa: X\to X $ be an intertwiner of the two actions.
Using
$*$-invariant form $\langle~,~\rangle_X^*$ on $X$, one defines a
$\bullet$-invariant form $\langle~,~\rangle_X^\bullet$ on $X$ by
\begin{equation}
\langle x,y\rangle^\bullet_X=\langle
x,\pi(t_{w_0})\kappa(y)\rangle_X^*, \quad x,y\in X.
\end{equation}
We verify that this form is indeed $\bullet$-invariant:
\begin{align*}
\langle
\pi(h)x,y\rangle_X^\bullet&=\langle\pi(h)x,\pi(t_{w_0})\kappa(y)\rangle_X^*=\langle
x,\pi(h^*)\pi(t_{w_0})\kappa(y)\rangle_X^*\\
&=\langle x,\pi(t_{w_0})\pi(\C
A(h^\bullet))\kappa(y)\rangle_X^*=\langle x,
\pi(t_{w_0})\kappa(\pi(h^\bullet)y)\rangle_X^*\\
&=\langle x, \pi(h^\bullet) y\rangle_X^\bullet.
\end{align*}
Since $X|_W$ is irreducible, the form $\langle~,~\rangle_X^\bullet$
may also be normalized so that it positive definite. Since
$\wti\om^\bullet=\wti\om,$ $\om\in V,$ it follows that every
$\pi(\wti\om)$ is symmetric, and thus it acts with real eigenvalues.

In conclusion, every $\pi(\wti\om)=0.$ The second claim is immediate
because the only relation between $\wti\om_1$ and $\wti\om_2$ is $[\wti\om_1,\wti\om_2]=-[p_{\om_1},p_{\om_2}].$

\end{proof}

\begin{remark}
Proposition \ref{p:2.1} implies that there is a one-to-one
correspondence between (hermitian) one $W$-type $\bH$-modules and
simple modules of the quotient algebra $\bC[W]/J,$ where $J$ is the
two-sided ideal generated by $\{[p_{\om_1},p_{\om_2}]: \om_1,\om_2\in
V\}.$ These simple modules were identified for each irreducible group $W$ in
\cite{BM2} and those results are used in the sequel. For example, for
$W=S_n$, they are parameterized by rectangular Young diagrams of shape
$d\times k$, $dk=n.$ It may be of independent interest to study the
ideal $J$ in more detail.
\end{remark}

\subsection{Tensoring with the reflection representation} Let $(\sigma, X)$ be an irreducible $W$-representation. One may also
ask if it is possible to extend $\sigma$ to a simple $\bH$-module, not
necessarily $*$-hermitian. Suppose there exists such a module
$(\pi,X).$ Then one can consider the vector subspace
\[Y=\{\pi(\wti\om)x:\ \om\in V_\bC,\ x\in X\}\subset X.\]
Using the commutation relation (\ref{e:propomtilde}), one sees
immediately that $Y$ is $W$-stable. If $Y=0$, we are in the same
setting as in Proposition \ref{p:2.1}, and thus $(\pi,X)$ is a
(hermitian) one $W$-module. 

Suppose now that $Y\neq 0.$ Since $(\sigma,X)$ is irreducible, it
follows that $Y=X.$ Moreover, there is a natural surjective map
\begin{equation}
\tau: X\otimes V_\bC\to Y,\quad \tau(x\otimes\omega)=\pi(\wti\om)x,
\end{equation}
which is $W$-equivariant with respect to the reflection representation
action on $V_\bC$:
\[
\pi(w)\tau(x\otimes\omega)=
\pi(w)\pi(\wti\om)x=\pi(\wti{w(\om)})\sigma(w)x=\tau(\sigma(w)x\otimes
{w(\om)})=\tau((\sigma\otimes\refl)(w)(x\otimes\omega)).
\]
By the assumptions, this map is nonzero, which implies the following lemma.

\begin{lemma}
A necessary condition for a $W$-type $\sigma$ to extend to a
non-hermitian $\bH$-module is that
$\Hom_W[\sigma\otimes\refl,\sigma]\neq 0.$
\end{lemma}

We now list the irreducible $W$-representations $\sigma$ with this
property. Recall that in type $A_{n-1}$, the $S_n$-types are parameterized
by partitions of $n$.

The irreducible $W(B_n)$-representations are parameterized by
bipartitions $(\lambda_L,\lambda_R)$. Denote the corresponding
representation by $\lambda_L\times \lambda_R$, see \cite{Ca}. Our
convention is that $n\times 0$ is the trivial representation, the
reflection representation is $(n-1)\times 1$, and
$0\times 1^n$ is the sign representation.

When
$\lambda_L\neq\lambda_R$, the restriction
$\lambda_L\times\lambda_R|_{W(D_n)}\cong
\lambda_R\times\lambda_L|_{W(D_n)}$ is irreducible. If
$\lambda_L=\lambda_R$, then the restriction splits into two
nonisomorphic irreducible $W(D_n)$-representations, $(\lambda_L\times\lambda_L)^\pm.$

For exceptional Weyl groups, we use the notation of \cite{Ca} for $W$-types.

\begin{proposition} Suppose the root system is irreducible.
\begin{enumerate}
\item If $w_0$ is central in $W$, then
  $\Hom_W[\sigma\otimes\refl,\sigma]=0$ for every $W$-type $\sigma.$
\item In type $A_{n-1}$, $\dim
  \Hom_W[\sigma_\lambda\otimes\refl,\sigma_\lambda]=d-1$, where $d$ is 
  the number of distinct parts in the partition $\lambda$. In particular,
  $\Hom_W[\sigma_\lambda\otimes\refl,\sigma_\lambda]=0$ if and only if
  $\lambda$ is a  partition of rectangular shape.
\item In type $D_n$, $n$ odd, $\dim
    \Hom_W[(\lambda_L\times\lambda_R)\otimes\refl,\lambda_L\times\lambda_R]=1$ if and only
    if $\lambda_R$ (viewed as a Young diagram) is obtained from
    $\lambda_L$ by removing one box; otherwise $
    \Hom_W[(\lambda_L\times\lambda_R)\otimes\refl,\lambda_L\times\lambda_R]=0$.
\item In type $E_6,$ $\dim\Hom_W[\sigma\otimes\refl,\sigma]=1$ if and
  only if $\sigma$ is one of the representations: \[(20,2), (20,20),
  (60,5), (60,8), (60,11), (64,4),(64,13),(81,6), (81,10), (90,8);\]
  otherwise $\Hom_W[\sigma\otimes\refl,\sigma]=0.$
\end{enumerate}
\end{proposition}

\begin{proof}
(1) When $w_0$ is central, the irreducible representations that appear in
$\sigma\otimes\refl$ have lowest harmonic degree of opposite parity to
the lowest harmonic degree of $\sigma.$ For type $B_n$, we also have
the following known rule for tensoring with the reflection
representation: 
\begin{equation}\label{e:reflB}
(\lambda_L\times\lambda_R)\otimes (n-1)\times 1=\sum_{(\lambda'_L,\lambda'_R)}\lambda'_L\times\lambda'_R,
\end{equation}
where $(\lambda'_L,\lambda'_R)$ are all possible bipartitions obtained from
$(\lambda_L,\lambda_R)$ by removing a box from one partition (diagram)
and adding it to the other. 

(2) This is an easy application of the Littlewood-Richardson
rule. More precisely,
\begin{equation}
\sigma_\lambda\otimes\refl+\sigma_\lambda=\sigma_\lambda\otimes \Ind_{S_{n-1}}^{S_n}(\triv)=\Ind_{S_{n-1}}^{S_n}(\sigma_\lambda|_{S_{n-1}}).
\end{equation}

(3) This follows immediately from (\ref{e:reflB}) and the restriction
rule from $W(B_n)$ to $W(D_n)$.

(4) We verified the statement directly using GAP 3.4.4 and the package 'chevie'.
\end{proof}

\begin{remark}
In type $A_{n-1}$, it is well-known that every $S_n$-type can be
lifted to a simple $\bH$-module. This is a consequence of the
existence of a surjective $\bC$-algebra homomorphism $\bH\to
\bC[S_n]$, see for example \cite[Lemma 3.2.1]{BC2}.
\end{remark}

\subsection{Dirac cohomology}\label{sec:3.2}

\begin{proposition}\label{p:3.2}
Let $(\pi,X)$ be a one $W$-type $\bH$-module. Then $H^D(X)=X\otimes S.$
\end{proposition}

\begin{proof}
Since $X$ is a unitary module, $\ker D\cap\im D=0,$ and $H^D(X)=\ker
D_X.$ But by Proposition \ref{p:2.1} and the definition of $D$, $D$ is
identically zero on $X\otimes S.$
\end{proof}

\begin{corollary}\label{c:3.2}
Let $X$ be a one $W$-type module and suppose $X|_W=\sigma$. If
$\wti\sigma$ is an irreducible $\wti W$-representation that occurs in
$\sigma\otimes S,$ for some spin $C(V)$-module $S$, then the central
character of $X$ is $\chi^{\wti\sigma}.$
\end{corollary}

\begin{proof}
This is immediate from Proposition \ref{p:3.2} and Theorem \ref{t:vogan}.
\end{proof}

We use Corollary \ref{c:3.2} to determine the central characters for
all one $W$-type modules. We refer to the case by case classification
in \cite{BM2}.

\subsection{Exceptional root systems}\label{sec:3.3}
For the exceptional root systems, we use the computer algebra system
GAP/chevie together with the character tables for $\wti W$-representations
\cite{Mo} to
decompose the tensor products $\sigma\otimes S$, where $\sigma$ is a
$W$-type which can be extended to a one $W$-type representation (the
explicit list is in \cite{BM2}), and $S$ a spin module. The results
are summarized in the tables below. The notation for
$W$-representations is as in \cite{Ca}, while the notation for
$\wti W$-representations is an in \cite{Mo}.

When the root system admits unequal parameters, we write $k_s$ and
$k_\ell$ for the value of the parameters of the short and long roots,
respectively. In the tables for $F_4$ and $G_2$, $\om_i$, $I=1,4$ and
$i=1,2$, respectively, denote the fundamental weights in $V$.

\begin{table}[h]\label{t:E6}
\caption{$E_6$}
\begin{tabular}{|c|c|c|}
\hline
$W$-type $\sigma$ &$\sigma\otimes S$ &$G\backslash\C N_{\text{sol}}$\\

\hline

$(1,0), (1,36)$ &$8_s$ &$E_6$\\
\hline
$(10,9)$ &$40_{ss}+20_s+20_{ss}$ &$D_4(a_1)$\\
\hline
$(15,4), (15,16)$ &$120_s$ &$E_6(a_3)$\\
\hline
$(24,6), (24,12)$ &$120_s+72_s$ &$E_6(a_3)$\\
\hline

\end{tabular}
\end{table}

\begin{table}[h]\label{t:E7}
\caption{$E_7$}
\begin{tabular}{|c|c|c|}
\hline
$W$-type $\sigma$ &$\sigma\otimes S$ &$G\backslash\C N_{\text{sol}}$\\

\hline

$(1,0), (1,63)$ &$8_s$ &$E_7$\\
\hline

$(15,7), (15,28)$ &$120_s$ &$E_7(a_4)$\\
\hline

$(21,3), (21,36)$ &$168_s$ &$E_7(a_2)$\\
\hline

$(35,4),(35,31)$ &$280_s$ &$E_7(a_3)$\\
\hline

$(70,9), (70,18)$ &$560_s$ &$E_7(a_5)$\\
\hline

$(84,12), (84,15))$ &$560_s+112_{ss}$ &$E_7(a_5)$\\
\hline

$(105,6), (105,15), (105,12), (105,21)$ &$120_s+720_s$ &$E_7(a_4)$\\

\hline

$(210,10), (210,21)$ &$112_{ss}+2\cdot 560_s+448_s$ &$E_7(a_5)$\\

\hline
\end{tabular}
\end{table}

\begin{table}[h]\label{t:E8}
\caption{$E_8$}
\begin{tabular}{|c|c|c|}
\hline
$W$-type $\sigma$ &$\sigma\otimes S$ &$G\backslash\C N_{\text{sol}}$\\

\hline

$(1,0),(1,120)$ &$16_s$ &$E_8$\\

\hline

$(50,8),(50,56)$ &$800_s$ &$E_8(b_4)$\\
\hline

$(84,4),(84,64)$ &$1344_{ss}$ &$E_8(a_3)$\\
\hline

$(175,12), (175,36)$ &$2800_{ss}$ &$E_8(b_6)$\\
\hline

$(525,12),(525,36)$ &$5600_{sss}+2800_s$ &$E_8(a_5)$\\
\hline

$(700,16),(700,28)$ &$8400_s+2800_{ss}$ &$E_8(b_6)$\\
\hline

$(972,12),(972,32)$ &$9072_s+6480_s$ &$E_8(a_6)$\\
\hline

$(168,24)$ &$1120_s+224_s+1344_s$ &$E_8(a_7)$\\
\hline

$(420,20)$ &$1120_s+2016_{ss}+2016_{sss}+224_s+1344_s$ &$E_8(a_7)$\\

\hline

$(448,9)$ &$7168_s$ &$E_8(b_5)$\\

\hline
\end{tabular}
\end{table}

\begin{table}[h]\label{t:F4}
\caption{$F_4$}
\begin{tabular}{|c|c|c|c|}
\hline
$W$-type $\sigma$ &$\sigma\otimes S$ &$G\backslash\C N_{\text{sol}}$
&$\chi^{\wti\sigma},k\not\equiv 1$\\

\hline

$(1,0), (1,24)$ &$4_s$ &$F_4$ &$k_\ell\om_1+k_\ell\om_2+k_s\om_3+k_s\om_4$\\
\hline
$(1,12)', (1,12)''$ &$4_{ss}$ &$F_4(a_3)$ &$k_\ell\om_1+k_\ell\om_2+(-2k_\ell+k_s)\om_3+k_s\om_4$\\
\hline
$(2,4)'', (2,16)'$ &$8_{sss}$ &$F_4(a_1)$ &$k_\ell\om_1+k_\ell\om_2+(-k_\ell+k_s)\om_3+k_s\om_4$\\
\hline
$(2,4)', (2,16)''$ &$8_{ssss}$ &$F_4(a_2)$ &$k_\ell\om_1+k_\ell\om_2+(-2k_\ell+k_s)\om_3+(3k_\ell-k_s)\om_4$\\
\hline
$(4,8)$ &$8_s+8_{ss}$ &$F_4(a_3)$ &$k_\ell\om_2+(-k_\ell+k_s)\om_4$\\
\hline
$(4,7), (4,7)'$ &$12_{ss}+4_{ss}$ &$F_4(a_3)$&no\\
\hline
$(6,6)'$ &$12_{ss}+8_s+4_{ss}$ &$F_4(a_3)$&no\\
\hline
$(8,3)', (8,9)''$ &$8_{ssss}+24_s$ &$F_4(a_2)$&no\\
\hline
\end{tabular}

\end{table}

\begin{table}[h]\label{t:G2}
\caption{$G_2$}
\begin{tabular}{|c|c|c|c|}
\hline
$W$-type $\sigma$ &$\sigma\otimes S$ &$G\backslash\C N_{\text{sol}}$
&$\chi^{\wti\sigma},k\not\equiv 1$\\
\hline
$(1,0),(1,6)$ &$2_s$ &$G_2$ &$k_\ell\om_1+k_s\om_2$\\
\hline
$(1,3)',(1,3)''$ &$2_{ss}$ &$G_2(a_1)$&$k_\ell\om_1+(-k_\ell+k_s)\om_2$\\
\hline
$(2,2)$ &$2_{ss}+2_{sss}$ &$G_2(a_1)$ &no\\
\hline

\end{tabular}

\end{table}

\subsection{Type $A_{n-1}$} For $A_{n-1}$, the only $S_n$-types that can be
extended to a ($*$-hermitian) one $W$-type module correspond to
partitions of rectangular shape \cite{BM2}. Let $\sigma_{d\times k}$ be the
irreducible $S_n$-representation parameterized by the rectangular
partition $(\underbrace{d,d,\dots,d}_k),$ where $dk=n.$ 

One says that a partition $\lambda$ of $n$ is strict if all the parts
of $\lambda$ are distinct. One says that $\lambda$ is even or odd if
$n-|\lambda|$ is even or odd, respectively; here $|\lambda|$ denotes the
number of parts in $\lambda.$ The classification of irreducible $\wti S_n$-representations goes back
to Schur. To every strict partition $\lambda$ of $n$, one constructs
one irreducible $\wti S_n$-representation $\wti\sigma_\lambda$, when
$\lambda$ is even, and two irreducible $\wti S_n$-representations
$\wti\sigma_\lambda^\pm$ when
$\lambda$ is odd. Moreover,
$\wti\sigma_\lambda^+\cong\wti\sigma_\lambda^-\otimes \sgn.$ These
representations are pairwise nonisomorphic and exhaust the dual of
$\wti S_n.$

As explained in \cite[Lemma 3.6.2]{BC2}, the tensor product rules from
\cite{St} imply that
\begin{equation}
\sigma_{d\times k}\otimes S^\ep=\begin{cases}2^{(k-1)/2}
  (\wti\sigma^+_{\text{hook}(d\times k)}+
  \wti\sigma^+_{\text{hook}(d\times k)}),&\text{ $k$ odd, $d$ even}\\2^{[k/2]}
  \wti\sigma^\ep_{\text{hook}(d\times k)},&\text{ otherwise.}
\end{cases}
\end{equation}
Here $\text{hook}(d\times k)=(d+k-1,d+k-3,\dots,|d-k|+1),$ and the
symbol $\ep$ stands for $+$ or $-$ if there exist two associate
representations of $\wti S_n$ or ``empty'' if there is only one.

Finally, the central character of the one $W$-type representation
supported on $\sigma_{d\times k}$ is one half the middle element of
the nilpotent orbit given in Jordan form by the strict partition $\text{hook}(d\times k).$

\subsection{Classical types}\label{sec:3.5}
We treat the case of the graded Hecke algebra of type $B_n$ with
parameters $k_\ell$ (on the long roots) and $k_s$ (on the short roots). There is an obvious isomorphism with the graded
Hecke algebra of type $C_n$ by changing the parameter function
appropriately. Also, as it is well-known, the graded Hecke algebra of type $B_n$ with
$k_s=0$ is isomorphic with the graded Hecke algebra of type $D_n$
extended by $\bZ/2\bZ.$

Set $\delta=2k_s/k_\ell$. By \cite[Proposition 3.24, Theorem
3.28]{BM2}, the $W(B_n)$-representations that can be extended to a one
$W$-type module are:
\begin{enumerate}
\item[(T1)] $\lambda_L=0$ or $\lambda_R=0$ (arbitrary parameters $k$);
\item[(T2)] $\lambda_L=d_1\times m_1$ and $\lambda_R=d_2\times m_2$ when $m_1-d_1=m_2-d_2+\delta.$
\end{enumerate}

The classification of irreducible $\wti{W(B_n)}$-representations was
obtained in \cite{Re}: for every partition $\lambda$ of $n$, the
representations
\begin{equation}
(\lambda\times 0)\otimes S,\text{ when $n$ is even}, \quad (\lambda\times 0)\otimes S^\pm,\text{ when $n$ is odd},
\end{equation}
are irreducible, pairwise nonisomorphic, and exhaust the dual of
$\wti{W(B_n)}$.

For one $W$-type modules of type (T1), i.e., the form $\lambda_L\times 0$ or
$0\times\lambda_R$, it is thus clear that the Dirac cohomology equals
\begin{equation}
(\lambda_L\times 0)\otimes S^\ep, \quad (\lambda_R^t\times 0)\otimes S^{-\ep},
\end{equation}
respectively. In particular, the central character of the one $W$-type
module is obtained by the combinatorial rule from \cite{C} explained at the end of
section \ref{sec:2.6}.

Thus, it remains to treat the case of one $W$-type modules of type
(T2). To determine the central character, it is sufficient, in light
of Corollary \ref{c:3.2}, to find one partition $\lambda$ on $n$ such
that
\begin{equation*}
\Hom_{\wti W}[(\lambda\times 0)\otimes
S,(\lambda_L\times\lambda_R)\otimes S]\neq 0, \text{ equivalently,} \Hom_{W}[(\lambda\times 0)\otimes
(S\otimes S^*),(\lambda_L\times\lambda_R)]\neq 0.
\end{equation*}
Since $S\otimes S\cong \wedge V_\bC$, when $n$ is even, $S^+\otimes
S^+\cong \wedge^{\text{even}} V_\bC$ and $S^+\otimes S^-\cong
\wedge^{\text{odd}} V_\bC$, when $n$ is odd, it is sufficient to find
$\lambda$ and $s$ such that
\begin{equation}
\Hom_W[(\lambda\times 0)\otimes \wedge^s V_\bC,
\lambda_L\times\lambda_R]\neq 0.
\end{equation}
Let $s$ equal the sum of parts of the partition $\lambda_R.$ Using Clifford theory for the semidirect product $W(B_n)=S_n\ltimes
(\bZ/2\bZ)^n$, we see that
\begin{equation}\label{e:cliffordB}
\lambda_L\times\lambda_R=\Ind_{S_{n-s}\times
  S_s\times (\bZ/2\bZ)^{n-s}\times (\bZ/2\bZ)^s}^{S_n}(\lambda_L\boxtimes\lambda_R\boxtimes \triv^{n-s}\boxtimes \sgn^s).
\end{equation}
On the other hand,
in bipartition notation, $\wedge^s V_\bC=(n-s)\times (1^s)$, and
\begin{equation}\label{e:tensB}
\begin{aligned}
&(\lambda\times 0)\otimes ((n-s)\times (1^s))=\\
&\Ind_{S_{n-s}\times
  S_s\times (\bZ/2\bZ)^{n-s}\times (\bZ/2\bZ)^s}^{S_n}((\lambda|_{S_{n-s}\times S_s}\otimes (\triv_{S_{n-s}}\boxtimes
\sgn_{S_s}))\boxtimes \triv^{n-s}\boxtimes \sgn^s).
\end{aligned}
\end{equation}
Comparing (\ref{e:cliffordB}) and (\ref{e:tensB}), we obtain:
\begin{lemma}
\begin{equation}\label{e:decomp}
(\lambda_L\times\lambda_R)\otimes S^\ep=\sum_{\lambda}
c_{\lambda_L,\lambda_R^t}^\lambda (\lambda\times 0)\otimes S^{\ep'},
\end{equation}
where
$c_{\lambda_L,\lambda_R^t}^\lambda=\dim\Hom_{S_n}[\lambda,\Ind_{S_{n-s}\times
S_s}^{S_n}(\lambda_L\boxtimes\lambda_R^t)]$ is the
Littlewood-Richardson coefficient, $s$ is the size of
$\lambda_R$. Here $\ep$ and $\ep'$ are ``empty''  if $n$ is
even, and when $n$ is odd, $\ep'=\ep$ is $s$ is odd and $\ep'=-\ep$ if
$s$ is even.
\end{lemma}
Every partition $\lambda$ that appears in the right hand side of
(\ref{e:decomp}) determines the central character of the one $W$-type
representation.

 Since in our case (T2), $\lambda_L$ and $\lambda_R$ are
rectangular partitions, the Littlewood-Richardson coefficients can be
explicitly described, see \cite{Ok}. Recall that $\lambda_L=d_1\times
m_1$ and $\lambda_R=d_2\times m_2$. Then $\lambda_R^t=m_2\times
d_2.$ Without loss of generality
suppose that $m_1\ge d_2.$ Then a 
partition $\lambda$ occurs in (\ref{e:decomp}) if and only if $\lambda$ has length at most
$m_1+d_2$, and if
$\lambda=(\lambda_1\ge\lambda_2\ge\dots\ge\lambda_{m_1+d_2}\ge 0)$,
then
\begin{enumerate}
\item[(i)] $\lambda_j+\lambda_{m_1+d_2-j+1}=d_1+m_2,$ for all
  $j=1,d_2$;
\item[(ii)] $\lambda_{d_2}\ge \max(d_1,m_2)$;
\item[(iii)] $\lambda_{d_2+1}=\lambda_{d_2+2}=\dots=\lambda_{m_1}=d_1$;
\end{enumerate}
moreover, for all these $\lambda$, the coefficient
$c_{\lambda_L,\lambda_R^t}^\lambda=1.$ For example, if $\lambda_v$ and
$\lambda_h$ are obtained by ``gluing'' the two rectangular shapes
$\lambda_L$ and $\lambda_R^t$ vertically and  horizontally,
respectively, then $\lambda_v$ and $\lambda_h$ both appear in
(\ref{e:decomp}). Consequently, any one determines the central
character by the procedure outlined at the end of section
\ref{sec:2.6}. It is easy to check that they give the same central
character if and only if the condition $m_1-d_1=m_2-d_2+\delta$ from (T2) is
satisfied. In fact, by  Corollary \ref{c:3.2}, we know that every
$\lambda$ satisfying the rules (i)-(iii) above must yield the same
central character!


\ifx\undefined\bysame
\newcommand{\bysame}{\leavevmode\hbox to3em{\hrulefill}\,}
\fi

\end{document}